\documentclass[12pt]{amsart}

\usepackage{amsthm, amssymb, amstext, amscd, amsfonts, amsxtra, latexsym, amsmath, comment, 
 mathrsfs, esint, stmaryrd}
\usepackage{color}  
\usepackage[svgnames]{xcolor}
\usepackage{fancybox}  
\usepackage{fullpage}
\usepackage[english]{babel}   
\usepackage[latin1]{inputenc}
 \usepackage{hyperref}
\hypersetup{colorlinks=true,linkcolor=blue,citecolor=magenta}
 
\numberwithin{equation}{section}  
\newtheorem{theorem}{Theorem}   
\numberwithin{theorem}{section}

\theoremstyle{remark}
 
\newcommand{\Z}{\mathbb{Z}}

\newcommand{\bs}{\boldsymbol}
\newcommand{\bea}{\begin{eqnarray}}
\newcommand{\eea}{\end{eqnarray}}
\newcommand{\be}{\begin{equation}} 
\newcommand{\ee}{\end{equation}}

\renewenvironment{proof}[1][Proof]{\begin{trivlist} \item[\hskip \labelsep {\bfseries #1:}]}{\qed\end{trivlist}}
\author{Kathrin Bringmann$^1$}
\address{$^1$ Mathematical Institute\\University of
Cologne\\ Weyertal 86-90 \\ 50931 Cologne \\Germany}
\email{kbringma@math.uni-koeln.de}
\thanks{The research of the first author was supported by the Alfried Krupp Prize for Young University Teachers of the Krupp foundation and the research leading to these results has received funding from the European Research Council under the European Union's Seventh Framework Programme (FP/2007-2013) / ERC Grant agreement n. 335220 - AQSER. The third author thanks the University of Cologne and the DFG for their generous support via the University of Cologne postdoc grant DFG Grant D-72133-G-403-151001011, funded under the Institutional Strategy of the University of Cologne within the German Excellence Initiative.}

\author{Jan Manschot$^2$}
\address{$^2$ School of Mathematics\\ Trinity College Dublin\\  Dublin 2\\Ireland}
\email{manschot@maths.tcd.ie}

\author{Larry Rolen$^3$}
\address{$^3$ Department of Mathematics\\
The Pennsylvania State University\\
University Park, PA 16802 
}
\email{larryrolen@psu.edu}

\begin{document}

\title{Identities for generalized Appell functions and the blow-up formula}  

\begin{abstract}
In this paper, we prove identities for a class of generalized Appell functions which are based on the $\operatorname{A}_2$ root lattice. The identities are reminiscent of periodicity relations for the classical Appell function, and are proven using only analytic properties of the functions. Moreover they are a consequence of the blow-up formula for generating functions of invariants of moduli spaces of semi-stable sheaves of rank 3 on rational surfaces. Our proof confirms that in the latter context, different routes to compute the generating function (using the blow-up formula and wall-crossing) do arrive at identical $q$-series. The proof also gives a clear procedure for how to prove analogous identities for generalized Appell functions appearing in generating functions for sheaves with rank $r>3$. 
\end{abstract} 

\maketitle

\section{Introduction}
Paul \'Emile Appell's study of doubly periodic functions led to the definition of the Appell function \cite{Appell:1886}
\begin{equation}
\label{eq:appell}
A(u,v)=A(u,v;\tau):=e^{\pi i u}\sum_{n\in \mathbb{Z}} \frac{(-1)^nq^{\frac{n(n+1)}{2}}e^{2\pi i n v}}{1-e^{2\pi i u}q^n},
\end{equation} 
with $u\in \mathbb{C}\backslash (\mathbb{Z}\tau+\mathbb{Z})$, $v\in \mathbb{C}$, $q:=e^{2\pi i \tau}, \tau \in \mathbb{H}$. Among the intriguing properties of this function are its behavior under elliptic transformations $(u,v)\mapsto (u+n_1\tau+m_1,v+n_2\tau+m_2)$, $n_1,n_2,m_1,m_2\in\Z$ and under modular transformations $(u,v,\tau)\mapsto (\frac{u}{c\tau+d},\frac{v}{c\tau+d},\frac{a\tau+b}{c\tau+d})$ with $\left(\begin{smallmatrix}a & b\\ c & d \end{smallmatrix}\right)\in \operatorname{SL}_2(\mathbb{Z})$. About one century ago, Srinivasa Ramanujan closely studied related $q$-hypergeometric series, and his discoveries led him to coin the term ``mock theta function" for such objects. The precise transformation properties of (\ref{eq:appell}) were only clarified in 2002 by Sander Zwegers  \cite{Zwegers:2002}. In particular, he showed that addition of a non-holomorphic integral to $A(u,v)$ gives a function which transforms as a multi-variable Jacobi form. 

Over the years various generalizations of the Appell function have arisen, for example higher level Appell functions and multi-variable Appell functions $A_{Q}(u,{\bs v})$ with  $Q$ an $n-$dimensional positive definite quadratic form and ${\bs v}\in \mathbb{C}^n$ \cite{Zwegers:2010}. The $A_Q(u,{\bs v})$ differ from $A(u,v)$ by having a more general quadratic form $Q$ in the numerator. Such functions occured as characters of Lie superalgebras \cite{Kac:2013, Semikhatov:2003uc} and also as certain generating functions in the context of rank 2 sheaves on rational surfaces \cite{Bringmann:2010sd, Vafa:1994tf, Yoshioka:1994}. More recently it was established that besides the possibly higher dimensional quadratic form, generalizations of (\ref{eq:appell}) with multiple terms in the denominator also occur as characters of Lie superalgebras \cite{Kac:2013} and for rank $>2$ sheaves \cite{Manschot:2014cca}. These new functions can be expressed as specializations of the following general shape (cf. \cite[equation (0.13)]{Kac:2013} and \cite[equation (4.2)]{Manschot:2014cca}):
\begin{equation} 
\label{eq:AQ} 
A_{Q, \{{\bs  m_j } \}} ({\bs u},{\bs v }) = A_{Q,\{{\bs m}_j\}}({\bs u},{\bs v};\tau):=\sum_{{\bs k}\in \Lambda }\frac{q^{\frac{1}{2}Q({\bs k})}e^{2\pi i {\bs v}\cdot {\bs k}}}{\prod_{\ell=1}^{n_-}(1-q^{ {\bs k}\cdot {\bs m}_\ell}e^{2\pi i u_\ell})},
\end{equation}  
where $Q$ is the quadratic form for the $n_+$-dimensional positive definite lattice $\Lambda$, ${\bs m}_j\in \Lambda^*$ with $\Lambda^*$ the lattice dual to $\Lambda$ and $j=1,\dots,n_-$. Furthermore, the argument ${\bs u}$ is a vector ${\bs u}=(u_1,\dots,u_{n_-})$ with $u_j\in \mathbb{C}\backslash (\mathbb{Z}\tau+\mathbb{Z})$, and ${\bs v}\in \Lambda^* \otimes \mathbb{C}\simeq \mathbb{C}^{n_+}$. In the context of Lie superalgebras, $\{ {\bs m}_j\}$ is a set of pairwise orthogonal vectors \cite{Kac:2013}, but for sheaves this is typically not the case. One can expand the terms in the denominator as geometric series, which leads to an $(n_++n_-)-$dimensional quadratic form of signature $(n_+,n_-)$. Therefore we naturally refer to (\ref{eq:AQ}) as an Appell function with signature $(n_+,n_-)$.   
 
The study of Appell functions with more general signatures than the mock modular objects in Zwegers' thesis is very natural. One expects that there is a more complicated notion of ``completion'' for such objects than is necessary for mock objects. Besides the two contexts mentioned above, these objects have arisen in a number of places in the recent literature on $q$-series; for example, the reader is referred to \cite{BRZ} and \cite{Raum} where it is explained how certain special cases of general type Appell sums can be reduced to products of mock modular objects, and what the relation of these Appell sums are to the theory of $H$-harmonic Maass-Jacobi forms and to the theory of Gromov-Witten potentials. Further examples of Appell-like sums with two factors in the denominator were studied recently by Ismail and Zhang (see Corollary 4.4 of \cite{IsmailZhang}) in the context of Rogers-Ramanujan type identities.

However, the general structure of these objects is still unknown (although results have been obtained in unpublished work of Raum and of Zagier-Zwegers). In this paper, we take a first step towards understanding such objects, which was also the first step in the proof of the modularity results of \cite{BRZ}. Namely, one first considers the error to satisfying the elliptic transformation satisfied by Jacobi forms, and writes this obstruction to transforming as a Jacobi forms in terms of simpler mock-type objects. Here we establish such formulas for our special cases.  A similar analysis could also be performed on many other examples of higher-index Appell sums using the methods used here. 
   
These special cases are instances of (\ref{eq:AQ}) which occur as building blocks of generating functions of topological invariants of moduli spaces of sheaves with rank $r$ on a ruled or rational surface $S$ \cite{Manschot:2014cca}. Expressing such generating functions in terms of generalized Appell functions is a major simplification compared to previously determined generating functions for $r=3$, and also allows one to write completely explicit generating functions for $r>3$ \cite{Manschot:2014cca}. For the surface $\mathbb{P}^2$ and rank $r$ sheaves, the quadratic form $Q$ in (\ref{eq:AQ}) corresponds to the one of the $\operatorname{A}_{r-1}$ root lattice. Moreover, they are naturally divided by classical theta functions of the $\operatorname{A}_{r-1}$ root lattice, such that we arrive at an $\operatorname{A}_{r-1}$ generalization of the Appell-Lerch sum $\mu(u,v):=\frac{A(u,v)}{\vartheta(v)}$ \cite{Zwegers:2002}, {with $\vartheta(z)$ defined in (\ref{eq:varth})}.
 
The identities proved here are implied by the blow-up formula.  This is a well-known formula in algebraic geometry which relates the generating function of Poincar\'e polynomials of semi-stable rank $r$ sheaves on a complex algebraic surface $S$ and the generating function for the blow-up $\phi: \widetilde S\to S$ with exceptional divisor $C_{\rm e}$ \cite{Gottsche:1998, Li:1999, Yoshioka:1996}. Here, we briefly recall the blow-up formula (see \cite{Manschot:2011ym} for more details and references). Let $H_{r,c_1}(z,\tau; S, J)=H_{r,c_1}(z; S, J)$ be the generating function of stack invariants of semi-stable sheaves on the algebraic surface $S$. The generating function sums over the second Chern class of the sheaves and keeps the rank $r$ and first Chern class $c_1$ fixed. The sheaves are $\mu$-semi-stable with respect to the polarization $J$. Then the blow-up formula states:
\be 
\label{blowup}
H_{r,\widetilde{c}_1}\!\left(z;\widetilde{S},\phi^* J\right)= \frac{b_{r,k}(z)}{\eta^r}\,H_{r,c_1}(z; S, J)
\ee
with $\widetilde c_1:=\phi^*c_1-kC_{\rm e}$, and where $b_{r,k}(z)$ and $\eta$ are defined by (\ref{eq:brk}) and (\ref{eq:eta}) respectively. Besides giving an intriguing relation between generating functions for $S$ and $\widetilde S$, the blow-up formula also gives interesting relations for $H_{r,\widetilde c_1}(z;\widetilde S,J)$, since for different $c_1$ and $k$ one can arrive at the same $\widetilde c_1=\phi^*c_1-kC_{\rm e}$. Since determining $H_{r,\widetilde{c}_1}(z;\widetilde{S},\phi^* J)$ gives rather different expressions depending on $\widetilde c_1$ \cite{Manschot:2011ym}, these relations imply surprising and intriguing identities for the generalized Appell functions. For $r=2$ and $S=\mathbb{P}^2$, these identities reduce to known periodicity relations of $A(u,v)$  \cite{Zwegers:2002}.

The identities for $r=3$ and $S=\mathbb{P}^2$ can be found in \cite[Section 4.3]{Manschot:2014cca}. We prove those identities using only analytic properties of the generating functions. This result is desirable from the point of view of $q$-series, since the intricate identities are proven without referring to an underlying meaning of the coefficients as topological invariants of moduli spaces, and further improves the understanding of these complicated objects. The result is also of interest for algebraic geometry, since it shows that the different ``routes'' for calculating $H_{3,c_1}(z;\mathbb{P}^2)$ do lead to identical generating functions \cite{Manschot:2011ym}. It therefore confirms consistency of all ingredients (i.e., suitable polarization, wall-crossing and blow-up formula) going into the calculations. The proof also indicates how to prove similar identities for the generalized Appell functions for $\operatorname{A}_{r-1}$ lattices following from equation (\ref{blowup}) for $r>3$.

\subsection*{Statement of results} 
Let $\vartheta$ be the Jacobi theta function, defined in (\ref{eq:varth}). We also denote by $b_{r,k}$ certain theta functions summing over the root lattice $\operatorname{A}_{r-1}$ (see (\ref{eq:brk}) for the definition). We further let $w:=e^{2\pi i z}$, $z\in \mathbb{C}\backslash (  \tau(\frac{\mathbb{Z}}{3}+\frac{\mathbb{Z}}{2})+(\frac{\mathbb{Z}}{3}+\frac{\mathbb{Z}}{2}))$. 
 
Our main results are the identities given in the following theorems.
Our first identity expresses the difference of two generalized Appell functions with signature $(2,2)$ in terms of higher level Appell functions, the eta function and Jacobi theta function. 
\begin{theorem}\label{lident1}
We have 
\begin{align*}
&\sum_{k_1,k_2\in \mathbb{Z}}\frac{w^{-k_1-2k_2+2}q^{k_1^2+k_2^2+k_1k_2+k_1+k_2}}{(1-w^{2}q^{2k_1+k_2})(1-w^{2}q^{k_2-k_1})}-\sum_{k_1,k_2\in \mathbb{Z}}\frac{w^{-k_1-2k_2+1}q^{k_1^2+k_2^2+k_1k_2+2k_1+2k_2+1}}{(1-w^{2}q^{2k_1+k_2+1})(1-w^{2}q^{k_2-k_1})}\\
&\hspace{-1mm} +\frac{i\eta^3}{\vartheta(2z)}\left(\sum_{k\in \mathbb{Z}}
   \frac{w^{-3k+2}q^{3k^2+2k}}{1-w^{3}q^{3k}}+ \sum_{k\in \mathbb{Z}}
   \frac{w^{-3k+1}q^{3k^2+k}}{1-w^{3}q^{3k}}- \sum_{k\in \mathbb{Z}}
 \frac{w^{-3k+1}q^{3k^2+4k+1}}{1-w^{3}q^{3k+1}}- \sum_{k\in \mathbb{Z}} \frac{w^{-3k+2}q^{3k^2-k}}{1-w^{3}q^{3k-1}}\right)\\
&=\frac{\eta^6\, \vartheta(z)}{\vartheta(2z)^2\,\vartheta(3z)}.
\end{align*}
\label{Iident1}
\end{theorem}
\noindent Our second identity is similar to the first, but differs in that the various terms are divided by a cubic theta function $b_{3,k}(z)$.
\begin{theorem}
\label{Id2} 
We have
\begin{align*}
&\frac{1}{b_{3,0}(z)}\sum_{k_1,k_2\in \mathbb{Z}}\frac{w^{-k_1-2k_2+3}q^{k_1^2+k_2^2+k_1k_2-\frac{1}{3}}}{(1-w^{2}q^{2k_1+k_2-1})(1-w^{2}q^{k_2-k_1})} -\frac{1}{b_{3,1}(z)}\sum_{k_1,k_2\in \mathbb{Z}}\frac{w^{-k_1-2k_2+1}q^{k_1^2+k_2^2+k_1k_2+2k_1+2k_2+1}}{(1-w^{2}q^{2k_1+k_2+1})(1-w^{2}q^{k_2-k_1})} \\
&+\frac{i\eta^3}{\vartheta(2z) b_{3,0}(z)}\left(\sum_{k\in \mathbb{Z}}
  \frac{w^{-3k+3}q^{3k^2-\frac{1}{3}}}{1-w^{3}q^{3k-1}}+\sum_{k\in \mathbb{Z}}
  \frac{w^{-3k}q^{3k^2+3k+\frac{2}{3}}}{1-w^{3}q^{3k+1}}\right)\\
&-\frac{i\eta^3}{\vartheta(2z) b_{3,1}(z)}\left( \sum_{k\in \mathbb{Z}}
\frac{w^{-3k+1}q^{3k^2+4k+1}}{1-w^{3}q^{3k+1}}+ \sum_{k\in \mathbb{Z}} \frac{w^{-3k+2}q^{3k^2-k}}{1-w^{3}q^{3k-1}}\right)=0.
\end{align*}
\end{theorem}
\noindent Our third identity is similar to the previous first two identities.
\begin{theorem}\label{Id3}
We have   
\begin{align*}
&\frac{1}{b_{3,0}(z)}\sum_{k_1,k_2\in \mathbb{Z}}\frac{w^{-k_1-2k_2}q^{k_1^2+k_2^2+k_1k_2}}{(1-w^{2}q^{2k_1+k_2})(1-w^{2}q^{k_2-k_1})} -\frac{1}{b_{3,1}(z)}\sum_{k_1,k_2\in \mathbb{Z}}\frac{w^{-k_1-2k_2-1}q^{k_1^2+k_2^2+k_1k_2+k_1+k_2+\frac{1}{3}}}{(1-w^{2}q^{2k_1+k_2+1})(1-w^{2}q^{k_2-k_1})}\\ &+\frac{2i\eta^3}{\vartheta(2z) b_{3,0}(z)}\sum_{k\in \mathbb{Z}}
  \frac{w^{-3k}q^{3k^2}}{1-w^{3}q^{3k}}-\frac{i\eta^3}{\vartheta(2z) b_{3,1}(z)}\left( \sum_{k\in \mathbb{Z}}
\frac{w^{-3k+1}q^{3k^2-2k+\frac{1}{3}}}{1-w^{3}q^{3k-1}}+ \sum_{k\in \mathbb{Z}} \frac{w^{-3k-1}q^{3k^2+2k+\frac{1}{3}}}{1-w^{3}q^{3k+1}}\right)\\
&=\frac{\eta^6\, \vartheta(z)}{\vartheta(2z)^2\,\vartheta(3z)\,b_{3,0}(z)} .
\end{align*}
\end{theorem}
The paper is organized as follows. In Section 2, we review the definitions and a few properties of the Dedekind eta function and the theta functions which occur in the statement of the main results. We give the proofs of these results in Section 3. 

\section{Preliminaries} 
This section defines the eta and theta functions which appear in the theorems and the proofs, and recalls the elliptic transformations of the theta functions.
The Dedekind eta and Jacobi theta functions are defined by 
\begin{eqnarray}
&&\eta=\eta(\tau):=q^{\frac{1}{24}}\prod_{n=1}^\infty(1-q^n), \label{eq:eta}\\
&&\vartheta(z) =\vartheta(z; \tau):=i w^{\frac12}q^{\frac18}(q)_\infty (w q)_\infty \left(w^{-1} \right)_\infty
=i\sum_{r\in \frac{1}{2}+\mathbb{Z}}(-1)^{r-\frac{1}{2}}q^{\frac{r^2}{2}}w^r. \label{eq:varth}
\end{eqnarray}
Here $(a)_n = (a;q)_n := \prod_{j=0}^{n-1}\left(1- a q^j\right)$.
The Dedekind eta function is well-known to be a modular form of weight $1/2$, whereas the Jacobi theta function is a Jacobi form of weight and index $1/2$. 

In particular, we need the elliptic transformations of $\vartheta$ $(n\in\Z)$:
\begin{equation}\label{genshift}
\vartheta(z+n\tau)=(-1)^nq^{-\frac{n^2}{2}}w^{-n}\vartheta(z).
\end{equation} 
Moreover note that
\[
\vartheta' (0) =-2 \pi \eta^3 .
\]
We next define the theta functions $b_{r,k}$ which appeared in equation (\ref{blowup}). They sum over the $\operatorname{A}_{r-1}$ root lattices and are defined by
\be
b_{r,k}(z)=b_{r,k}(z;\tau):=\sum_{\sum_{j=1}^ra_j=0\atop a_j\in \frac{k}{r} + \mathbb{Z}} q^{-\sum_{j<\ell}a_ja_\ell}w^{\frac{1}{2}\sum_{j<\ell}\left(a_j-a_\ell\right)}. \label{eq:brk}
\ee 
We are mainly interested in the case $r=3$, the  cubic theta functions, when the theta functions are explicitly given by
\begin{eqnarray*}
&& b_{3,0}(z)=\sum_{m,n\in  \Z}
 q^{m^2+n^2+mn} w^{m+2n}, \\
&&  b_{3,1}(z)=\sum_{m,n \in \Z}
 q^{m^2+n^2+mn+m+n+\frac{1}{3}} w^{m+2n+1}.
\end{eqnarray*} 
These satisfy the following transformation formula
\begin{equation}\label{bshift}   
b_{3,k}(z+\tau)=q^{-1}w^{-2}b_{3,k}(z).  
\end{equation}
These ``cubic'' theta functions have been of great interest in the literature, starting with their introduction by the Borweins in \cite{Borwein}. In particular, the general theory of theta functions of shapes such as those in $b_{3,k}$ in a general context analogous to the usual Jacobian elliptic functions has been worked out in great detail by Schultz in \cite{Schultz}.

\section{Proofs of theorems}
\label{sec:proofs}

\begin{proof}[Proof of Theorem \ref{lident1}]
The idea is to show that the difference of the left and right-hand sides satisfies the elliptic transformation law of a Jacobi form of negative index but has no poles, and thus must vanish. Denote the right-hand and the left-hand sides of the claimed equation by $r$ and $\ell$, respectively. Then by \eqref{genshift}, we find that
$$
r(z+\tau) =q^8w^{16}r(z)
.
$$
We now aim to show that $\ell$ satisfies the same transformation law. Write $\ell=\ell_1+\ell_2$ with $\ell_1$ the first and $\ell_2$ the second line of the equation defining $\ell$. 

We first consider $\ell_1$. We call the first summand (resp. negative of the second) summand in $\ell_1$, $\ell_{11}$ (resp.\ $\ell_{12}$). We get
$$
\ell_{11} (z+\tau) -q^8 w^{16} \ell_{11}(z) = \sum_{k_1,k_2\in\Z}\frac{w^{-k_1-2k_2+6}q^{k_1^2+k_2^2+k_1k_2-2k_1-5k_2+8}}{\left(1-w^2q^{2k_1+k_2}\right)\left(1-w^2q^{k_2-k_1}\right)}\left(1-w^{12}q^{3k_1+6k_2}\right),
$$
where for the first summand, we shifted $k_2 \mapsto k_2 -2$.
We then substitute the identity
\begin{align}\label{decompose}
1-w^{12}q^{3k_1+6k_2}&=\left(1-w^2 q^{k_2-k_1}\right)\left(1+w^4 q^{k_1+2k_2}+w^8 q^{2k_1+4k_2}\right) \\
&\hspace{5mm} +\left(1-w^2 q^{2k_1+k_2}\right)\left(w^2 q^{k_2-k_1}+w^6 q^{3k_2}+w^{10} q^{k_1+5k_2}\right). \notag
\end{align}
The contribution from the first term, which we call $G_1$, gives
$$
G_1(z)=\sum_{k_1,k_2\in\Z}\frac{w^{-3k_1-2k_2+6}q^{3k_1^2+k_2^2+3k_1k_2-8k_1-5k_2+8}}{1-w^2q^{k_2}}
\left(1+w^4 q^{3k_1+2k_2} +w^8 q^{6k_1+4k_2 }\right),
$$
where we let $k_2\mapsto k_2-2k_1$ and then $k_1\mapsto -k_1$. 
Similarly the contribution from the second term in (\ref{decompose}), denoted $G_2$, is
\[
G_2(z)=\sum_{k_1,k_2\in\Z}\frac{w^{-3k_1-2k_2+8}q^{3k_1^2+k_2^2+3k_1k_2-7k_1-4k_2+8}}{1-w^2q^{k_2}}
\left(1+w^4 q^{3k_1+2k_2} +w^8 q^{6k_1+4k_2}\right),
\]
where we sent $k_2\mapsto k_2+k_1$.

Next
$$
\ell_{12} (z+\tau) -q^{8} w^{16} \ell_{12}(z)=\sum_{k_1,k_2\in\Z}\frac{w^{-k_1-2k_2+5}q^{k_1^2+k_2^2+k_1k_2-k_1-4k_2+6}}{\left(1-w^2q^{2k_1+k_2+1}\right)\left(1-w^2q^{k_2-k_1}\right)}\left(1-w^{12}q^{3k_1+6k_2+3}\right),
$$
where in the first term, we shifted $k_2 \mapsto k_2 -2$.
Similarly as before, we decompose
\begin{align*}
1-w^{12}q^{3k_1+6k_2+3}&=\left(1-w^2q^{k_2-k_1}\right) \left(1+w^4q^{k_1+2k_2+1}+w^8q^{2k_1+4k_2+2}\right)\\
&\hspace{20mm} +\left(1-w^2q^{2k_1+k_2+1}\right) \left(w^2q^{k_2-k_1}+w^6q^{3k_2+1}+w^{10}q^{k_1+5k_2+2}\right).
\end{align*}
We call the contribution from the first term $G_3$. After changing variables to $k_2 \mapsto k_2 -2k_1-1$ and then $k_1 \mapsto -k_1-1$, we get
\be 
\label{G3}
G_{3}(z)=  
\sum_{k_1,k_2\in \mathbb{Z}}\frac{w^{-3k_1-2k_2+4}q^{3k_1^2+k_2^2+3k_1k_2-4k_1-3k_2+4}}{1-w^{2}q^{k_2}}\left(1+w^4q^{3k_1+2k_2+2}+w^{8}q^{6k_1+4k_2+4}\right).
\ee
The contribution of the second summand $G_4$ is, shifting $k_2 \mapsto k_2 +k_1$,
\be
\label{G4}
G_{4}(z)=\sum_{k_1,k_2\in \mathbb{Z}}\frac{w^{-3k_1-2k_2+7}q^{3k_1^2+k_2^2+3k_1k_2-5k_1-3k_2+6}}{1-w^{2}q^{k_2}}\left(1+w^4q^{3k_1+2k_2+1}+w^{8}q^{6k_1+4k_2+2}\right).
\ee

Substituting the expressions for $G_j$ we find that
\begin{align*}
&G_1(z)+G_2(z)-G_3(z)-G_4(z)=\sum_{k_1, k_2 \in\Z}\frac{w^{-3k_1-2k_2}q^{3k_1 ^2 +k_2 ^2 +3k_1 k_2+8}}{1-w^2q^{k_2}} \Big (w^6 q^{-8k_1-5k_2} \\
& +w^{10} q^{-5k_1-3k_2} +w^{14} q^{-2k_1-k_2}+w^8 q^{-7k_1-4k_2}+w^{12} q^{-4k_1-2k_2}+w^{16} q^{-k_1}- w^4 q^{-4k_1-3k_2-4}\\
& -w^{8} q^{-k_1-k_2-2} -w^{12} q^{2k_1+k_2} - w^7 q^{-5k_1-3k_2-2}-w^{11} q^{-2k_1-k_2-1}-w^{15} q^{k_1+k_2} \Big ) .
\end{align*}
We label the twelve terms in the brackets by $j=1,2,\dots,12$. For $j=1,2,4$ and 10, we shift $k_1\mapsto k_1+1$ and for $j=9$ we shift $k_1\mapsto k_1-1$. This gives
\begin{align*} 
\sum_{k_1, k_2 \in\Z}&\frac{w^{-3k_1-2k_2}q^{3k_1 ^2 +k_2 ^2 +3k_1 k_2+8}}{1-w^2q^{k_2}}\Big (w^3 q^{-2k_1-2k_2-5}+w^{7} q^{k_1-2}+w^{14} q^{-2k_1-k_2}+w^5 q^{-k_1-k_2-4} \\
&+w^{12} q^{-4k_1-2k_2}+w^{16} q^{-k_1}- w^4 q^{-4k_1-3k_2-4}-w^{8} q^{-k_1-k_2-2}-w^{15} q^{-4k_1-2k_2+1}\\
& - w^4 q^{k_1-4}-w^{11} q^{-2k_1-k_2-1}-w^{15} q^{k_1+k_2} \Big ) .
\end{align*}
We denote the terms in the brackets of the last expression by $\mathcal{D}_j$. Substituting now 
\begin{align*}
\mathcal{D}_1&=w^3 q^{-2k_1-2k_2-5}\left(1-w^2q^{k_2}+w^2q^{k_2}\right),\hspace{5mm} \mathcal{D}_6=w^{14} q^{-k_1-k_2}\left(w^2q^{k_2}-1+1\right),\\
 \mathcal{D}_7&=-w^4 q^{-4k_1-3k_2-4}\left(1-w^2q^{k_2}+w^2q^{k_2}\right), \hspace{5mm} \mathcal{D}_{12}=-w^{13} q^{k_1}\left(w^2q^{k_2}-1+1\right) 
\end{align*}
and dividing out in these terms $1-w^2q^{k_2}$, we obtain 
\begin{equation}\label{terms}
\begin{aligned} 
\sum_{k_1, k_2 \in\Z}&\frac{w^{-3k_1-2k_2}q^{3k_1 ^2 +k_2 ^2 +3k_1 k_2+8}}{1-w^2q^{k_2}}\Big (w^5 q^{-2k_1-k_2-5}+w^{7} q^{k_1-2} +w^{14} q^{-2k_1-k_2} +w^5 q^{-k_1-k_2-4} \\
&\hspace{15mm}+w^{12} q^{-4k_1-2k_2}+w^{14} q^{-k_1-k_2}- w^6 q^{-4k_1-2k_2-4}-w^{8} q^{-k_1-k_2-2}  \\
&  \hspace{15mm}-w^{15} q^{-4k_1-2k_2+1} - w^4 q^{k_1-4}-w^{11} q^{-2k_1-k_2-1}-w^{13} q^{k_1} \Big )  \\
&\hspace{-7mm}+\sum_{k_1, k_2 \in\Z} w^{-3k_1-2k_2}q^{3k_1 ^2 +k_2 ^2 +3k_1 k_2+8}\Big ( w^3 q^{-2k_1-2k_2-5}-w^{14} q^{-k_1-k_2}-w^4 q^{-4k_1-3k_2-4}+w^{13} q^{k_1}\Big ).
\end{aligned}
\end{equation}
The last line is a sum of four theta functions $T_j$ ($j=1,\dots,4$), which we show vanishes. To this end, substitute in $T_1$ and $T_2$, $k_1\mapsto -k_1$ and then $k_2\mapsto k_2+3k_1$. Next replacing in $T_1$, $k_1\mapsto k_1+1$, and $k_2\mapsto k_2-2$ shows that $T_1+T_3=0$, and replacing in $T_2$, $k_1\mapsto k_1+1$ and $k_2\mapsto k_2-1$, shows that $T_2+T_4=0$. Grouping the terms of the other lines of (\ref{terms}) gives our final expression for 
\[
\ell_1(z+\tau)-q^8w^{16}\ell_1(z)=G_1(z)+G_2(z)-G_3(z)-G_4(z),
\]
namely
\begin{eqnarray} 
\label{diffl1}
&&\left(1-w^3q^2+w^9q^4\right) \sum_{k_1, k_2 \in\Z}\frac{w^{-3k_1-2k_2}q^{3k_1 ^2 +k_2 ^2 +3k_1 k_2+8}}{1-w^2q^{k_2}}\left(w^5q^{-k_1-k_2-4}-w^4q^{k_1-4}\right)\nonumber\\
&&+\left(1-w^6q^4+w^9q^5\right) \sum_{k_1, k_2 \in\Z}\frac{w^{-3k_1-2k_2}q^{3k_1 ^2 +k_2 ^2 +3k_1 k_2+8}}{1-w^2q^{k_2}}\left(w^5q^{-2k_1-k_2-5}-w^6q^{-4k_1-2k_2-4}\right).\nonumber
\end{eqnarray}

Next we consider the second line $\ell_2$, and set
$$
\lambda_2(z):=\frac{-i\vartheta(2z)}{\eta^3}\ell_2(z).
$$
Using (\ref{genshift}) we thus aim to compute the combination
\begin{equation}\label{lstardiff}
L_2(z):=\lambda_2(z+\tau)-q^6w^{12}\lambda_2(z).
\end{equation}
We do this termwise and only carry out the details for the first term. This yields the following contribution to (\ref{lstardiff})
\begin{multline*}
\label{1stterm}
\sum_{k\in\Z}\frac{w^{-3k+2}q^{3k^2-k+2}}{1-w^3q^{3k+3}} -\sum_{k\in\Z}\frac{w^{-3k+14}q^{3k^2+2k+6}}{1-w^3q^{3k}}
=\sum_{k\in\Z}\frac{w^{-3k+5}q^{3k^2-7k+6}}{1-w^3q^{3k}} \left(1-w^9q^{9k}\right),\\
= \sum_{k\in\Z}w^{-3k+2}q^{3k^2-k+2} +\sum_{k\in\Z}w^{-3k+8}q^{3k^2-4k+6}+\sum_{k\in\Z}w^{-3k+11}q^{3k^2-k+6},
\end{multline*}
where we shifted $k\mapsto k-1$ in the first summand on the first line. To obtain the second line we used
\[
\left(1-w^9 q^{9k}\right) =\left(1-w^3 q^{3k}\right)\left(1+w^3 q^{3k} +w^6 q^{6k}\right),
\] 
and mapped $k\mapsto k+1$ in the first summand.
 Treating the other terms similarly, we obtain that
\begin{align}
\label{L2z}
 L_2 (z)&=\left(1 -w^3 q^2 +w^{9} q^4 \right)\sum_{k\in\Z} \left(w^{-3k+2}q^{3k^2-k+2}-w^{-3k+4}q^{3k^2-5k+4}\right)
 \\
 &\hspace{5mm}+\left(1 -w^6 q^4+w^{9} q^5 \right)\sum_{k\in\Z} \left(w^{-3k+1}q^{3k^2-2k+1}-w^{-3k+2}q^{3k^2-4k+2}\right).\nonumber
\end{align}

Summarizing, we have to show the following identity:
\[
\label{eq:redident}
\ell_1(z+\tau)-q^8w^{16}\ell_1(z)+q^2w^4\frac{i\eta^3}{\vartheta(2z)}L_2(z)=0.
\]
Since both $\ell_1(z+\tau)-q^8w^{16}\ell_1(z)$ (see \eqref{diffl1}) and $L_2$ (see \eqref{L2z}) contain a term with $(1 -w^3 q^2 +w^{9} q^4)$ and $(1 -w^6 q^4+w^{9} q^5)$, the identity in particular holds if the following two identities hold:
\begin{multline}\label{2.11}
\sum_{k_1, k_2 \in\Z} \frac{w^{-3k_1-2k_2+4}q^{3k_1 ^2 +k_2 ^2 +3k_1 k_2+4}}{1-w^2q^{k_2}} \left(wq^{-k_1-k_2}-q^{k_1}\right)\\
=\frac{i\eta^3}{\vartheta(2z)} \sum_{k\in\Z}\left(w^{-3k+8}q^{3k^2-5k+6}  -w^{-3k+6}q^{3k^2-k+4} \right),
 \end{multline}
\begin{multline}\label{2.12}
\sum_{k_1, k_2 \in\Z} \frac{w^{-3k_1-2k_2+5}q^{3k_1 ^2 +k_2 ^2 +3k_1 k_2+3}}{1-w^2q^{k_2}} \left(q^{-2k_1-k_2}-wq^{-4k_1-2k_2+1}\right)\\
=\frac{i\eta^3}{\vartheta(2z)}\sum_{k\in\Z}\left(w^{-3k+6}q^{3k^2-4k+4}   -w^{-3k+5}q^{3k^2-2k+3}  \right). 
 \end{multline} 
Denote the left-hand side of (\ref{2.11}) by $\mathcal{L}_1$ and the right-hand side by $\mathcal{R}_1$, and similarly define $\mathcal{L}_2$ and $\mathcal{R}_2$ for (\ref{2.12}). We again use elliptic transformations to prove these identities. We obtain for the shift $z\mapsto z+\tau$ of the $\mathcal{R}_j$'s, using \eqref{genshift} and the shift $k\mapsto k+1$ in the first summand, 
\[
\mathcal{R}_1 (z+\tau)=-w^4q^8\mathcal{R}_2(z), \qquad \mathcal{R}_2 (z+\tau)=-wq^4\mathcal{R}_1(z).
\]

This implies that we have to prove the following identities for the left hand sides:
\begin{eqnarray}\label{ident31}
\mathcal{L}_1(z+\tau)+w^4q^8\mathcal{L}_2(z)=0,\qquad \mathcal{L}_2(z+\tau)+w q^4\mathcal{L}_1(z)=0. 
\end{eqnarray}
In $\mathcal{L}_1(z+\tau)$ we change $k_1 \mapsto k_1 +1$ and $k_2 \mapsto k_2 -2$. After adding $w^4q^8 \mathcal{L}_2(z)$, we can divide out by $1-w^2q^{k_2}$ to obtain
\begin{align*}
&\mathcal{L}_1(z+\tau)+w^4q^8\mathcal{L}_2(z)\\
&=\sum_{k_1, k_2 \in\Z} w^{-3k_1-2k_2+1}q^{3k_1 ^2 +k_2 ^2 +3k_1 k_2+5}\left( w^5q^{-4k_1-4k_2+7}\left(1+w^2q^{k_2}\right)-w^4q^{-2k_1-3k_2+6}\left(1+w^2q^{k_2}\right)\right). 
\end{align*}
This is a sum of four theta functions $S_j$.
A direct calculation shows that $S_1=-S_3$ and $S_2=-S_4$, which proves the first identity of (\ref{ident31}). The second identity in (\ref{ident31}) is proven similarly.

To finish the proofs of (\ref{2.11}) and (\ref{2.12}), we are left to show (since we have the elliptic transformation law of negative index Jacobi forms) that both sides have the same poles and the same residues.

We start with $\mathcal{R}_1$ which has {at most} simple poles for $z\in\frac12\Z+\tau\frac12\Z$.

Firstly consider the pole at $z=0$. We compute the residue as
$$
-\frac{i}{4\pi}\sum_{k\in\Z}\left(q^{3k^2-5k+6}  -q^{3k^2-k+4} \right)=0,
$$
where we shifted in the first summand $k\mapsto -k+1$ to find that the residue vanishes.  At $z=1/2$, we obtain the residue
$$
\frac{i}{4\pi}\sum_{k\in\Z} \left((-1)^{k}q^{3k^2-5k+6}  -(-1)^{k}q^{3k^2-k+4} \right)=-\frac{i}{2\pi}\sum_{k\in\Z}(-1)^kq^{3k^2-k+4},
$$
where we let $k\mapsto -k+1$ in the first summand. 
At $z=\tau/2$, the residue is
$$
\frac{iq^{\frac{1}{2}}}{4\pi}  \sum_{k\in\Z}\left(q^{-\frac{3k}{2}+4}q^{3k^2-5k+6}  -q^{-\frac{3k}{2}+3}q^{3k^2-k+4} \right)
=\frac{i}{4\pi}\sum_{k\in\Z} \left( q^{3k^2+\frac{k}{2}+7} - q^{3k^2-\frac{5k}{2}+\frac{15}{2}}\right),
$$
shifting again $k\mapsto -k+1$ in the first summand. Moreover, we used that
\[
\vartheta' (\tau ) =-q^{-\frac{1}{2}} \frac{\partial}{\partial z} \left[ w^{-1} \vartheta (z) \right]_{z=0} = -q^{-\frac{1}{2}} \vartheta' (0)
\]
since $\vartheta$ is odd as a function of $z$. Finally at $z=\tau/2+1/2$, our residue becomes
\begin{multline*}
-\frac{iq^{\frac{1}{2}}}{4\pi} \sum_{k\in\Z}\left( (-1)^k q^{-\frac{3k}{2}+4} q^{3k^2-5k+6}  - (-1)^kq^{-\frac{3k}{2}+3}q^{3k^2-k+4} \right)\\
=\frac{i}{4\pi}\sum_{k\in\Z} \left( (-1)^kq^{3k^2+\frac{k}{2}+7} +(-1)^kq^{3k^2-\frac{5k}{2}+\frac{15}{2}}\right),
\end{multline*}
making again the shift $k\mapsto -k+1$ in the first summand. 

We next turn to $\mathcal{L}_1$. {We see that we again have at most simple poles at $\frac{1}{2} \mathbb{Z} + \tau \frac{1}{2}\mathbb{Z}$}. For $z=0$ a pole {\bf can} only come from the term $k_2=0$ yielding the contribution
$$ 
\frac{i}{4\pi}\sum_{k_1\in\Z}q^{3k_1^2+4}\left(q^{-k_1}-q^{k_1}\right)=0  
$$
by changing $k_1\mapsto-k_1$ in the second term. At $z=1/2$ again the pole comes from $k_2=0$ yielding
$$
\frac{i}{4\pi}\sum_{k_1\in\Z}(-1)^{k_1}q^{3k_1^2+4}\left(-q^{-k_1}-q^{k_1}\right)=-\frac{i}{2\pi}\sum_{k_1\in\Z}(-1)^kq^{3k_1^2-k_1+4}
$$
which matches the residue of $\mathcal{R}_1$. At $z=\tau/2$, the pole comes from $k_2=-1$ yielding
$$
\frac{i}{4\pi}\sum_{k_1\in\Z}q^{3k_1^2-\frac{9k_1}{2}+8}\left(q^{-k_1+\frac32}-q^{k_1}\right)
=\frac{i}{4\pi}\left(\sum_{k_1\in\Z}q^{3k_1^2+\frac{k_1}{2}+7}-\sum_{k_1\in\Z}q^{3k_1^2-\frac{5k_1}{2}+\frac{15}{2}}\right),
$$
where for the first summand  we changed $k_1\mapsto k_1+1$, and in the second summand $k_1\mapsto -k_1+1$. 

Finally at $z=\tau/2+1/2$ the pole comes from $k_2=-1$ yielding 
\begin{align*}
&-\frac{i}{4\pi}\sum_{k_1\in\Z}(-1)^{k_1}q^{3k_1^2-\frac{9k_1}{2}+8}\left(q^{-k_1+\frac32}+q^{k_1}\right)\\
&\hspace{20mm}=\frac{i}{4\pi}\left(\sum_{k_1\in\Z}(-1)^{k_1}q^{3k_1^2+\frac{k_1}{2}+7}+\sum_{k_1\in\Z}(-1)^{k_1}q^{3k_1^2-\frac{5k_1}{2}+\frac{15}{2}}\right),
\end{align*}
where for the first summand we shifted again $k_1\mapsto k_1+1$, and in the second summand $k_1\mapsto -k_1+1$. So we see that all residues of $\mathcal{L}_1$ and $\mathcal{R}_1$ match. The calculation of the residues of $\mathcal{L}_2$ and $\mathcal{R}_2$ is similar to those $\mathcal L_1$ and $\mathcal R_1$, so we omit the calculation.

At this point, we have proven that the left and right hand sides of Theorem \ref{Iident1} transform identically under the elliptic transformation $z\mapsto z+\tau$. It is straightforward to check that they also transform the same under $z\mapsto z+1$. To finish the proof, one needs again to verify that the poles and residues of both sides match. To this end, we multiply both sides by $\frac{\vartheta(2z)^2}{\eta^6}$. We see that both sides still satisfy the elliptic transformation law of negative index Jacobi forms with simple poles at $\frac{m}{3}\tau+\frac{n}{3}$, with $(m,n)=(0,1)$, $(0,2)$, $(1,0)$, $(1,1)$, $(1,2)$, $(2,0)$, $(2,1)$ and $(2,2)$. Note that at $(0,0)$ we have a removable singularity. We omit the details of the calculations and only list the residues. For $(m,n)=(0,1)$, the residue equals $\frac{\vartheta(\frac13)}{6\pi \eta^3}$; for $(m,n)=(0,2)$, $-\frac{\vartheta(\frac13)}{6\pi \eta^3}$; for $(m,n)=(1,0)$, $\frac{q^{\frac12}\vartheta(\frac{\tau}{3})}{6\pi \eta^3}$; for $(m,n)=(1,1)$, $-\frac{q^{\frac12}\vartheta(\frac{\tau+1}{3})}{6\pi \eta^3}$; for $(m,n)=(1,2)$, $\frac{q^{\frac12}\vartheta(\frac{\tau+2}{3})}{6\pi \eta^3}$; for $(m,n)=(2,0)$, $-\frac{q^{2}\vartheta(\frac{2\tau}{3})}{6\pi \eta^3}$; for $(m,n)=(2,1)$, $\frac{q^{2}\vartheta(\frac{2\tau+1}{3})}{6\pi \eta^3}$;
and for $(m,n)=(2,2)$, $-\frac{q^{2}\vartheta(\frac{2\tau+2}{3})}{6\pi \eta^3}$.

\end{proof} 

\begin{proof}[Proof of Theorem \ref{Id2}]
We follow the same strategy as in the proof for Theorem \ref{lident1}. Since many steps are almost identical to those in the proof of Theorem \ref{Iident1}, we only give the main intermediate results. We denote the first line of Theorem \ref{Id2} by $\ell^*_{1}(z)$. First we consider the periodicity of the generalized Appell functions in $\ell^*_{1}(z)$ under $z\mapsto z+\tau$ without the $b_{3,k}$. We define $\ell^*_{11}$ as $b_{3,0}$ times the first summand in $\ell^*_{1}$. Sending $ w \mapsto wq$ in $\ell^*_{11}$ and adding $-q^8 w^{16}$ times the original function, we have
\begin{eqnarray} 
&& \label{diffA1} \qquad \sum_{k_1,k_2\in \mathbb{Z}}\frac{w^{-k_1-2k_2+7}q^{k_1^2+k_2^2+k_1k_2-3k_1-6k_2+\frac{32}{3}}}{(1-w^{2}q^{2k_1+k_2-1})(1-w^{2}q^{k_2-k_1})}\left(1-w^{12}q^{3k_1+6k_2-3}\right).
\end{eqnarray}
Similar to the proof of Theorem \ref{Iident1}, we substitute the equality
\begin{align} 
\label{identdiff}
1-w^{12}q^{3k_1+6k_2-3}=&\left(1-w^2q^{k_2-k_1}\right)\left(1+w^4q^{k_1+2k_2-1}+w^8q^{2k_1+4k_2-2}\right)\nonumber\\
&\hspace{20mm}+\left(1-w^2q^{2k_1+k_2-1}\right)\left(w^2q^{k_2-k_1}+w^6q^{3k_2-1}+w^{10}q^{k_1+5k_2-2}\right), 
\end{align} 
and write (\ref{diffA1}) as $G^*_{1}+G^*_{2}$, where $G^*_{1}$ and $G^*_{2}$ come from the two terms in (\ref{identdiff}). They are given by  
\begin{eqnarray*} 
&&G^*_{1}(z)=\sum_{k_1,k_2\in \mathbb{Z}}\frac{w^{-3k_1-2k_2+5}q^{3k_1^2+k_2^2+3k_1k_2-6k_1-4k_2+\frac{17}{3}}}{1-w^{2}q^{k_2}}\left(1+w^4q^{3k_1+2k_2+1}+w^8q^{6k_1+4k_2+2}\right),\\
&&G^*_{2}(z)=\sum_{k_1,k_2\in \mathbb{Z}}\frac{w^{-3k_1-2k_2+6}q^{3k_1^2+k_2^2+3k_1k_2-3k_1-2k_2+\frac{14}{3}}}{1-w^{2}q^{k_2}}\left(1+w^4q^{3k_1+2k_2+2}+w^{8}q^{6k_1+4k_2+4}\right),
\end{eqnarray*} 
where for $G^*_1$ we shifted $k_2\mapsto k_2-2k_1+1$ followed by $k_1\mapsto -k_1$, and for $G^*_2$ we shifted $k_2\mapsto k_2+k_1$ followed by $k_1\mapsto k_1+1$. 
After a few more elementary manipulations, we see that we can write $G^*_{1}+G^*_{2}$ as  
\begin{eqnarray}
\label{G12G22} 
&&\left(1+w^6q^4+w^9q^5\right)\sum_{k_1,k_2\in\Z} \frac{w^{-3k_1-2k_2+4}q^{3k_1^2+k_2^2+3k_1k_2+\frac{8}{3}}}{1-w^2q^{k_2}} \nonumber \\
&&+\left(1+w^3q^2+w^9q^4\right) \sum_{k_1,k_2\in\Z} \frac{w^{-3k_1-2k_2+8}q^{3k_1^2+k_2^2+3k_1k_2-3k_1-k_2+\frac{14}{3}}}{1-w^2q^{k_2}}\\
&&+wq^{\frac{5}{3}}\left(1+w^3q^2+w^6q^4\right)b_{3,0}(z). \nonumber
\end{eqnarray} 

We define $\ell^*_{12}$ as $-b_{3,1}$ times the second summand in $\ell^*_{1}$. Comparing with Theorem \ref{Iident1}, we see that $\ell^*_{12}$ is identical to $\ell_{12}$. Therefore, $\ell^*_{12}(z+\tau)-q^8w^{16}\ell^*_{12}(z)$ equals $G_3+G_4$ with $G_3$ and $G_4$ given respectively in (\ref{G3}) and (\ref{G4}). With a few manipulations, we bring $G_{3}+G_{4}$ in the following form:
\begin{align} 
&\left(1+w^9q^4\right) \sum_{k_1,k_2\in\Z} \frac{w^{-3k_1-2k_2+9}q^{3k_1^2+k_2^2+3k_1k_2-5k_1-2k_2+6}}{1-w^2q^{k_2}} \nonumber \\
&+\left(1+w^9q^5\right) \sum_{k_1,k_2\in\Z} \frac{w^{-3k_1-2k_2+6}q^{3k_1^2+k_2^2+3k_1k_2-4k_1-2k_2+4}}{1-w^2q^{k_2}} \label{G32G42} \\
&+ w^6q^4\sum_{k_1,k_2\in\Z} \frac{w^{-3k_1-2k_2+5}q^{3k_1^2+k_2^2+3k_1k_2-2k_1-k_2+3}}{1-w^2q^{k_2}}  \nonumber \\
&+w^3q^2\sum_{k_1,k_2\in\Z} \frac{w^{-3k_1-2k_2+7}q^{3k_1^2+k_2^2+3k_1k_2-k_1+4}}{1-w^2q^{k_2}} +wq^{\frac{5}{3}}\left(1+w^3q^2+w^6q^4\right)\,b_{3,1}(z), \nonumber
\end{align}
where we pulled out the factors $w^6q^4$ and $w^3q^2$ for later convenience. 
 
We continue with the second line of Theorem \ref{Id2}, which we denote by $\ell^*_{2}$. 
Defining 
$$
\lambda^*_{2}(z):=\frac{-i\vartheta(2z)b_{3,0}(z)}{\eta^3}\ell^*_{2}(z),
$$
we determine for the shift $z\mapsto z+\tau$ that
\begin{align}
\label{ell22}
&\lambda^*_{2}(z+\tau)-w^{12}q^6\lambda^*_{2}(z)\nonumber \\
&\hspace{2mm}=\left(1+w^3q^2+w^9q^4\right)\sum_{k\in \Z}w^{-3k+3}q^{3k^2-3k+ \frac{8}{3}}+\left(1+w^6q^4+w^9q^5\right)\sum_{k\in \Z} w^{-3k}q^{3k^2+\frac{2}{3}}.
\end{align}
Similarly, we denote the third line by $\ell^*_{3}$ and define $\lambda^*_{3}(z):=\frac{i\vartheta(2z)b_{3,1}(z)}{\eta^3}\ell^*_{3}(z)$. We obtain
\begin{eqnarray}
\label{ell32}
\lambda^*_{3}(z+\tau)-w^{12}q^6\lambda^*_{3}(z)&=&\left(1+w^9q^4\right)\sum_{k\in \Z} w^{-3k+1}q^{3k^2+k+2}+w^6q^4\sum_{k\in\Z}w^{-3k+1}q^{3k^2-2k+1}\nonumber \\
&&+\left(1+w^9q^5\right)\sum_{k\in \Z} w^{-3k+2}q^{3k^2-4k+2}+w^3q^2\sum_{k\in\Z}w^{-3k+2}q^{3k^2-k+2}.
\end{eqnarray}

Combining these results, we can again find identities which imply that the required periodicity holds as in Theorem \ref{Iident1}. To determine the identities, first note that the last terms of (\ref{G12G22}) and (\ref{G32G42}) cancel each other in the computation of $\ell^*_{1}(z+\tau)-q^9w^{18}\ell^*_{1}(z)$.  The first identity then follows by combining the sums $\Sigma_{k_1,k_2\in \mathbb{Z}}$ in (\ref{G12G22}),  (\ref{G32G42}), (\ref{ell22}), and (\ref{ell32}) which are multiplied $(1+w^9q^4)$ on the left: 
\begin{eqnarray}\label{FirstId}
\nonumber&&\frac{1}{b_{3,0}(z)}\sum_{k_1,k_2\in\Z} \frac{w^{-3k_1-2k_2+8}q^{3k_1^2+k_2^2+3k_1k_2-3k_1-k_2+\frac{14}{3}}}{1-w^2q^{k_2}}\\
&&-\frac{1}{b_{3,1}(z)}\sum_{k_1,k_2\in\Z} \frac{w^{-3k_1-2k_2+9}q^{3k_1^2+k_2^2+3k_1k_2-5k_1-2k_2+6}}{1-w^2q^{k_2}}\\
\nonumber&& +w^4q^2 \frac{i\eta^3}{\vartheta(2z)} \left( \frac{1}{b_{3,0}(z)} \sum_{k\in \Z}w^{-3k+3}q^{3k^2-3k+\frac{8}{3}} - \frac{1}{b_{3,1}(z)} \sum_{k\in \Z}w^{-3k+1}q^{3k^2+k+2}\right)=0.
\end{eqnarray} 
The second identity follows by combining the sums $\Sigma_{k_1,k_2\in \mathbb{Z}}$ multiplied by $(1+w^9q^5)$
\begin{eqnarray*}
&&\frac{1}{b_{3,0}(z)}\sum_{k_1,k_2\in\Z} \frac{w^{-3k_1-2k_2+4}q^{3k_1^2+k_2^2+3k_1k_2+\frac{8}{3}}}{1-w^2q^{k_2}} \\
&&-\frac{1}{b_{3,1}(z)}\sum_{k_1,k_2\in\Z} \frac{w^{-3k_1-2k_2+6}q^{3k_1^2+k_2^2+3k_1k_2-4k_1-2k_2+4}}{1-w^2q^{k_2}}\\
&& +w^4q^2 \frac{i\eta^3}{\vartheta(2z)} \left( \frac{1}{b_{3,0}(z)} \sum_{k\in \Z}w^{-3k}q^{3k^2+\frac{2}{3}} - \frac{1}{b_{3,1}(z)} \sum_{k\in \Z}w^{-3k+2}q^{3k^2-4k+2}\right)=0.
\end{eqnarray*} 
The third identity follows by combining the sums $\Sigma_{k_1,k_2\in \mathbb{Z}}$ multiplied by $w^6 q^4$
\begin{eqnarray*}
&&\frac{1}{b_{3,0}(z)}\sum_{k_1,k_2\in\Z} \frac{w^{-3k_1-2k_2+4}q^{3k_1^2+k_2^2+3k_1k_2+\frac{8}{3}}}{1-w^2q^{k_2}}\\
&&-\frac{1}{b_{3,1}(z)}\sum_{k_1,k_2\in\Z} \frac{w^{-3k_1-2k_2+5}q^{3k_1^2+k_2^2+3k_1k_2-2k_1-k_2+3}}{1-w^2q^{k_2}}\\
&& +  w^4q^2\frac{i\eta^3}{\vartheta(2z)}\left( \frac{1}{b_{3,0}(z)} \sum_{k\in \Z}w^{-3k}q^{3k^2+\frac{2}{3}} - \frac{1}{b_{3,1}(z)} \sum_{k\in \Z}w^{-3k+1}q^{3k^2-2k+1}\right)=0.
\end{eqnarray*} 
The fourth identity follows by combining the sums $\Sigma_{k_1,k_2\in \mathbb{Z}}$ multiplied by $w^3 q^2$
\begin{eqnarray*}  
&&\frac{1}{b_{3,0}(z)}\sum_{k_1,k_2\in\Z} \frac{w^{-3k_1-2k_2+8}q^{3k_1^2+k_2^2+3k_1k_2-3k_1-k_2+\frac{14}{3}}}{1-w^2q^{k_2}}\\
&&-\frac{1}{b_{3,1}(z)}\sum_{k_1,k_2\in\Z} \frac{w^{-3k_1-2k_2+7}q^{3k_1^2+k_2^2+3k_1k_2-k_1+4}}{1-w^2q^{k_2}}\\
&& +w^4q^2 \frac{i\eta^3}{\vartheta(2z)} \left( \frac{1}{b_{3,0}(z)} \sum_{k\in \Z}w^{-3k+3}q^{3k^2-3k+\frac{8}{3}} - \frac{1}{b_{3,1}(z)} \sum_{k\in \Z}w^{-3k+2}q^{3k^2-k+2}\right)=0.
\end{eqnarray*}
We prove the first identity \eqref{FirstId} 
and leave the proofs of the other identities as an exercise for the reader. First we argue that the function satisfies the elliptic transformation law of a negative index Jacobi form. To this end, consider the change of variables $z\mapsto z+2\tau$. For the third term of the left-hand side of \eqref{FirstId} we get after letting $z\mapsto z+2\tau$ and
using \eqref{genshift} and \eqref{bshift} that
\begin{eqnarray*}
&&\frac{w^4q^{10}}{b_{3,0}(z+2\tau)} \frac{i\eta^3}{\vartheta(2z+4\tau)}\sum_{k\in \mathbb{Z}}w^{-3k+3}q^{3k^2-9k+\frac{26}{3}} \overset{k\mapsto k+1}{=}\frac{w^{16}q^{22}}{b_{3,0}(z)}\frac{i\eta^3}{\vartheta(2z)}\sum_{k\in\mathbb{Z}}w^{-3k}q^{3k^2-3k+\frac{5}{3}},
\end{eqnarray*}    
which is $w^9q^{20}$ times the original third term. One can show similarly that the first term plus the second term, and the fourth term, are multiplied by $w^9q^{20}$ under $z\mapsto z+2\tau$. It is also easy to check that the two sides have the same behavior under $z\mapsto z+1$. Therefore, the left-hand side satisfies the elliptic transformation law of a Jacobi form of negative index. If we multiply \eqref{FirstId} by $b_{3,0}b_{3,1}$, we still have a function which satisfies the elliptic transformation law of a Jacobi form of negative index, since $b_{3,k}(z+2\tau)=w^{-4}q^{-4}b_{3,k}(z)$. Under the transformation $z\mapsto z+\tau$, \eqref{FirstId} becomes the third identity multiplied by $q^5$. Thus, we can work as before with a vector-valued identity, and consider the periodicity $z\mapsto z+\tau$. 
  
Since the index is negative, what remains is to show that the residues within the parallelogram spanned by $1$ and $\tau$ of the left-hand side of \eqref{FirstId} vanish. This is a tedious, but straightforward proof, which we omit.  The proof of the Identities 2, 3, 4,  and, consequently, Theorem \ref{Id2} follow similarly.
\end{proof}

\begin{proof}[Proof of Theorem \ref{Id3}]
The proof of Theorem \ref{Id3} is almost identical to the proof of Theorem \ref{Id2}, and so we omit the details.
\end{proof}

\providecommand{\href}[2]{#2}\begingroup\raggedright\endgroup

\end{document}